\documentclass[12pt]{article}
\usepackage{amsmath}
\usepackage{amssymb}
\usepackage{latexsym}
\usepackage{enumerate}
\usepackage[ansinew]{inputenc}
\usepackage[all]{xy}

\newtheorem{theorem}{Theorem}[section]
\newtheorem{proposition}[theorem]{Proposition}
\newtheorem{corollary}[theorem]{Corollary}
\newtheorem{lemma}[theorem]{Lemma}
\newtheorem{example}[theorem]{Example}
\newtheorem{remark}[theorem]{Remark}
\newtheorem{question}[theorem]{Question}
\newtheorem{definition}[theorem]{Definition}
\newenvironment{proof}{\noindent{\sc Proof.}}{\quad\qed\medskip}

\newcommand{\Z}{{\mathbb Z}}

\newcommand{\F}{{\cal F}}
\newcommand{\E}{{\cal E}}

\newcommand{\Hom}{{\rm Hom}}
\newcommand{\Aut}{{\rm Aut}}

\newcommand{\End}{{\rm End}}

\newcommand{\Ker}{{\rm Ker}\,}

\newcommand{\Gal}{{\rm Gal}}
\newcommand{\coGal}{{\rm coGal}}

\newcommand{\im}{{\rm Im}\,}
\newcommand{\Id}{{\rm Id}}

\newcommand{\qed}{\quad\lower0.05cm\hbox{$\Box$}}

\newcommand{\downarrowright}[1]{\downarrow
\rlap{\raise0.1cm\hbox{$\scriptstyle{#1}$}}}

\newcommand{\downarrowleft}[1]{\rlap{\kern-0.2cm
\raise0.1cm\hbox{$\scriptstyle{#1}$}}\downarrow}

\newcommand{\uparrowright}[1]{\uparrow
\rlap{\lower0.1cm\hbox{$\scriptstyle{#1}$}}}

\newcommand{\uparrowleft}[1]{\rlap{\kern-0.2cm
\lower0.1cm\hbox{$\scriptstyle{#1}$}}\uparrow}

\newcommand{\ra}{\rightarrow}

\newcommand{\epi}{\mbox{$\to$\hspace{-0.35cm}$\to$}}

\begin{document}
\setcounter{page}{1}
\title{Envelopes and covers for groups}
\footnotetext{


 The first-named author was supported by grant MTM2013-46837-P and FEDER funds and grant 18394/JLI/13 by the Fundaci\'on S\'eneca-Agencia de Ciencia y Tecnolog\'{\i}a de la Regi\'on de Murcia in the framework of III PCTRM 2011-2014. Second-named author was supported by projects MTM2010-15831 and MTM2013-42178-P, funded by the Spanish Ministry of Economy and Competitiveness and by the Junta de Andalucia Grant FQM-213.}
\author{Sergio Estrada and Jos\'e L. Rodr\'{\i}guez}
\date{June 2nd, 2016}
\maketitle

\begin{abstract}
We connect work done by Enochs, Rada and Hill in module approximation theory with work undertaken
by several group theorists and algebraic topologists in the context of
homotopical localization and cellularization of spaces.
This allows one to consider envelopes and covers of arbitrary groups.
We show some characterizing results for certain classes of groups,  and
present some open questions.
\end{abstract}


\setcounter{equation}{0}

\section{Introduction}

A group homomorphism $\varphi: H \to G$ is called a {\it localization}  if it induces a bijection $\varphi^*:\End(G) \cong \Hom(H,G)$ given by $\varphi^*(f)=f\varphi$. Dually, it is called a {\it cellular cover} (or {\it co-localization}) if it induces a bijection $\varphi_*:\End(H) \cong \Hom(H,G)$, given by $\varphi_*(f)=\varphi f$.
Localizations and cellular covers of groups have been broadly studied
in recent years, motivated by their implications in homotopical localization theory. Special attention has been given to understanding whether some properties or structures are preserved under localization or cellular covers in the category of groups. An updated treatment of this subject (initiated in \cite{Cas00} and \cite{Lib00}) can be found in the recent book by G\"obel and Trlifaj \cite{GT12}
(also see references in \cite{CRT}).

In a few cases it is possible to give an explicit list of localizations of a fixed group $H$ or cellular covers of a fixed group $G$. For example, see the recent work by Blomgren, Chach\'olski, Farjoun, and Segev \cite{BCFS13}
where a complete classification of all cellular covers of each finite simple group is given.

However, in many other cases the classification is not possible, as we may obtain a proper class of solutions.
The use of infinite combinatorial principles, like Shelah's Black Box
and its relatives, has allowed one to produce either arbitrarily large localizations or cellular covers for certain groups. For instance, in \cite{GRS02} and \cite{GS02} the authors constructed large localizations of finite simple groups. Countable as well as arbitrarily large cellular covers of cotorsion-free abelian groups with given ranks have also been constructed (see \cite{BD06}, \cite{FGSS07} \cite{GRS12}). On the other hand, interesting new results have been achieved in \cite{GHS11} where G\"obel, Herden, and Shelah constructed absolute $E$-rings (localizations of $\Z$) of a size below the first Erd\"os cardinal. This approach has yielded the solution of an old problem by Fuchs. 

By relaxing the uniqueness property, some of the previous constructions could be adapted to find new envelopes and covers of groups. More precisely,  $\varphi: H \to G$ is an {\it envelope} if  $\varphi^*:\End(G) \to \Hom(H,G)$ is surjective and every endomorphism $f:G\to G$ such that $f\varphi= \varphi$ is an automorphism. On the other hand, $\varphi:H\to G$ is a {\it cover} if $\varphi_*:\End(H) \to \Hom(H,G)$ is surjective and  every endomorphism $f:H\to H$ such that $\varphi f= \varphi$ is an automorphism (see Section 2 for  more details). These notions would have interesting applications in related areas such as homotopical localization or module approximation theory (see Remark \ref{otherterminologies}).

Our motivation to write this article was to connect notions and tools from different contexts.
In particular, our aim is to connect the work by Enochs and Rada \cite{E05}, as well as Hill \cite{Hill} (where they considered
torsion free covers of abelian groups having trivial co-Galois group) to the work by  Buckner--Dugas \cite{BD06} and Farjoun--G\"obel--Shelah--Segev \cite{FGSS07}. Our Theorem \ref{trivial-co-galois} ensures
that, in fact, $\F$-covers of arbitrary groups with trivial co-Galois group are cellular covers for any class $\F$ of groups. 

A similar result holds for envelopes under certain assumptions. We show in Theorem~\ref{trivial-galois}
that if $H \to G$ is an $\F$-envelope with trivial Galois group,
and either $G$ is nilpotent or $H$ is abelian then
$H\to G$ is a localization.
The proof  is a tricky play of iterated commutators combined with the properties of envelopes.

Section \ref{examples} provides some more results,
examples and counterexamples of envelopes of groups. It also includes several open questions for further work 
involving the group of integers, finite $p$-groups and finite simple groups.


\section{Envelopes and covers for groups}
\label{envelopes-covers}
In this section we adapt some well known notions of envelopes and covers
of modules to the category of groups.
This yields a weaker version of (co)-localization theory
(see e.g. \cite{Bou77}, \cite{F97}, \cite{Hi03}),
where the class of (co)-local groups is replaced by
an arbitrary class of groups $\F$,
which is not closed under retractions, nor (co)-limits in general.


We follow Enochs and Xu's terminology (see e.g. \cite{E81},  \cite{Xu96}).
Let $\F$ be a class of groups closed under isomorphisms.
A group homomorphism $\varphi:H\to G$ is called an
{\it $\F$-preenvelope} of $H$ if $G\in \F$ and for all $F\in \F$
the induced map $\varphi^*:\Hom(G,F)\to \Hom(H,F)$ given by $\varphi^*(f)=f\varphi$ is surjective.
If in addition, every endomorphism $f:G\to G$ such that
$f \varphi=\varphi$ is an automorphism, then
$\varphi$ is called an {\it $\F$-envelope}.
{\it $\F$-precovers} and {\it $\F$-covers} with respect to a class of groups
$\F$ are defined dually.

Let $\varphi: H\to G$ be an $\F$-envelope of a group $H$.
The {\it Galois group} of $\varphi$ is defined as the subgroup
$\Gal(\varphi)\subseteq \Aut(G)$ consisting of those automorphisms
$f$ of $G$ such that $f \varphi=\varphi$.
The {\it co-Galois group} $\coGal(\varphi)$ of an $\F$-cover $\varphi: H\to G$
is defined dually.

Note that $\F$-envelopes or $\F$-covers are determined up to isomorphism,
whenever they exist.
A class of groups $\F$ is said to be {\it (pre)enveloping} if
every group admits an $\F$-(pre)envelope. Dually, it is said to be
{\it (pre)covering} if every group admits an $\F$-(pre)cover.

\begin{remark}
\label{otherterminologies}
{\rm
\begin{enumerate}
\item In module theory, Auslander and Smal{\o} \cite{AS80} used the names 
{\it left $\F$-approximation} instead of $\F$-preenvelope and
{\it minimal left $\F$-approximation} instead of $\F$-envelope.
Preenveloping classes are also called {\it covariantly finite subcategories}.
Dual terms are used for precovering classes.

\item
In any category, {\it injective hulls}
with respect to a given class of morphisms $I$
(which contains the identity morphism)
correspond exactly to $\F$-envelopes, where $\F$
is the class of $I$-injective objects (i.e. those objects $X$
such that $\Hom(B,X) \to \Hom(A,X)$ is surjective
for all $A\to B$ in $I$).
The dual claim for covers also holds (cf. \cite[Proposition 3]{Porst81}).

\item In category theory a preenveloping class $\F$ is known as a {\it weakly reflective subcategory} (see e.g. \cite{He11}).
For this, we choose an $\F$-preenvelope $\eta: X \to RX$ for every object $X$. This is a {\it weak reflection} onto $\F$. Note that this choice need not be functorial.
In this context, {\it weak factorization systems} give rise to envelopes and covers, when one factorizes morphisms into the terminal object, and morphisms from the initial object, respectively.
\end{enumerate}
}
\end{remark}

In this paper we focus on envelopes and covers for
certain groups. We exploit the following notion:
\begin{definition}
{\rm Given a group $H$, we say that a homomorphism $\varphi:H\to G$
is an {\it envelope} (of $H$), if there exists a class of groups
$\F$ such that $\varphi$ is an $\F$-envelope.
Dually $\varphi:H\to G$ is a {\it cover} (of $G$) if there exists a class of groups $\F$ such that $\varphi$ is an $\F$-cover.
}
\end{definition}

Observe that \cite[Lemma 2.1]{Cas00} and its obvious dual,
have their corresponding analogues for envelopes and covers:

\begin{proposition}
\label{charac-env}
\begin{enumerate}\item
A homomorphism $\varphi: H\to G$ is an envelope of $H$  if and only if
$\varphi^*: \End(G)\to \Hom(H,G)$ is surjective and
the preimage of $\varphi$ is contained in $\Aut(G)$.

\item A homomorphism $\varphi: H\to G$ is a cover of $G$ 
if and only if
$\varphi_*: \End(H)\to \Hom(H,G)$ is surjective and the
preimage of $\varphi$ is contained in $\Aut(H)$.
\end{enumerate}
\end{proposition}
\begin{proof} For the converses, in (1) take $\F=\{G\}$ and in (2) take $\F=\{H\}$.
It is clear that $\varphi$ is an $\F$-envelope in (1), and an $\F$-cover in (2).
\end{proof}

\begin{example}
{\rm
The inclusion $\Z/p^r\to \Z/p^\infty$ of the cyclic group of
order $p^r$ into the Pr\"ufer group is an $\F$-envelope, where
$\F$ is the class of all injective abelian groups.
On the other hand, the canonical homomorphism $\Z/p^r\to \Z/p^s$ is an envelope for all $r, s\geq 1$,
with $\F=\{\Z/p^s\}$ (see more examples in Section \ref{examples}). The localizations of $\Z/p^r$ are exactly the projections $\Z/p^r\to \Z/p^s$ with $s\leq r$ (see e.g. \cite{CRT}).
}
\end{example}

As in the case of modules, we have the following reduction for groups:

\begin{proposition}\label{reducction-injective-surjective}
Let $\F$ be any class of groups.
If $\varphi: H\to G$ is an $\F$-envelope, then
$\im(\varphi)\hookrightarrow G$ is an $\F$-envelope
with the same Galois group.
Dually, if $\varphi: H\to G$ is an $\F$-cover, then
$H\twoheadrightarrow \im(\varphi)$ is also an $\F$-cover with the same co-Galois
group.
\qed\end{proposition}
On the other hand, surjective preenvelopes and monomorphic precovers are completely described by socles and radicals as we observe next.

 Recall that given a class $\E$ of group epimorphisms, the
{\it $\E$-socle} $S_\E (H)$ of a group $H$ is the
(normal) subgroup generated by $\psi(\Ker \varphi)$ where $\psi: E\to H$ and $\varphi: E \epi E'$ is in $\E$. If $\F$ is any class of groups and
$\E=\{\varphi:F\to 1,\ F\in {\mathcal F}\}$, then the $\E$-socle of $H$ is called the $\F$-socle of $H$.
The {\it $\E$-radical} $T_\E$ is the union
$T_\E(H)= \bigcup_i T^i$ of a continuous chain starting at
$T^0=S_\E(H)$ and defined inductively by $T^{i+1}/T^i=S_\E(H/T^{i})$,
and $T^\lambda=\bigcup_{\alpha<\lambda}  T^\alpha$ if $\lambda$ is a limit ordinal.
The quotient $H \epi H/T_\E (H)$ is the {\it epireflection} with respect to $\E$ (see e.g. \cite{RS00}).



\begin{proposition}
\label{socles-covers}
Let $\mathcal F$ be any class of groups. Then, the following are equivalent:
\begin{itemize}
\item[a) ] $H\hookrightarrow G$ is an $\mathcal F$-precover.
\item[b) ] $H\hookrightarrow G$ is an $\mathcal F$-cover having unique liftings.
\item[c) ] $H$ is the $\mathcal F$-socle of $G$.\qed
\end{itemize}
\end{proposition}

Recall some terminology of orthogonal pairs from localization theory. Given a class of groups $\F$ we denote by $^\perp \F$ the class of homomorphisms $g:A\to B$ that induce a bijection
$g^*:\Hom(B,F)\cong \Hom(A,F)$ for all $F\in \F$, where $g^*(h)=hg$. For a class $\E$ of homomorphisms,
$\E^\perp$ denotes the class of groups $G$ such that $g^*:\Hom(B,G)\cong \Hom(A,G)$ for all $g\in \E$.

\begin{proposition}
\label{radical-envelopes}
Let $\mathcal F$ be any class of groups. Then, the following are equivalent:
\begin{itemize}
\item[a) ] $H\twoheadrightarrow G$ is an $\mathcal F$-preenvelope.
\item[b) ] $H\twoheadrightarrow G$ is an $\mathcal F$-envelope having unique
liftings.
\item[c) ] $H\twoheadrightarrow G$ is an $\E$-epireflection onto $G\in \F$, where $\E=(^\perp \F)\cap Epi$.
\end{itemize}
\end{proposition}
\begin{proof}
Denote $f:H\twoheadrightarrow G$.
Clearly (a) is equivalent to (b), since $f$ is an epimorphism.
Suppose (b). Then $f$ is in $^\perp \F$ because $f$ is an $\F$-envelope with unique liftings, thus $f\in \E$.
On the other hand, $G\in \F \subset (^\perp \F)^\perp \subset \E^\perp$, hence $G$ is $\E$-local, and $f$ is the $\E$-epireflection, hence (c) holds. Conversely, suppose (c). This implies that $G$ is $\E$-local and $f$ is an $\E$-equivalence. Therefore, $f\in (^\perp \F)$, i.e. all groups in $\F$ are $f$-local. We conclude that $f$ is an $\F$-envelope, which shows (b).
\end{proof}

In all these cases the (co)-Galois groups are trivial, but this holds of course for arbitrary (co)-localizations (see \cite{Bou77} or \cite{Cas00} for notation). 

\begin{proposition}
\label{localizations-are-envelopes}
Let $(L,\eta)$ be a localization functor, and let
$\F$ be the class of $L$-local groups.
Then $\eta_H:H\to LH$ is an $\F$-envelope
with $Gal(\eta_H)=\{\Id_{LH}\}$ for every group $H$.

The dual result for co-localizations functors also holds.\qed
\end{proposition}

\section{Envelopes and covers with trivial Galois groups}
\label{envelopes-covers-trivial-Galois-group}

We next show the converse of Proposition \ref{localizations-are-envelopes} for covers, i.e.
that covers of groups with trivial co-Galois groups
are cellular covers (cf. Lemma 2.2 in \cite{E05} for modules).
This links work done by Enochs--Rada \cite{E05} and
Hill \cite{Hill} on torsion free covers of abelian groups with trivial co-Galois group, and recent work on cellular
covers by Chach\'olski--Farjoun--G\"obel--Segev \cite{CFGS07}, and Buckner--Dugas \cite{BD06}.

\begin{theorem}
\label{trivial-co-galois}
Let $\pi: H\to G$ be any cover of groups. Then the following
are equivalent:
\begin{enumerate}
\item[(a)] $\coGal(\pi)=\{{\rm Id}_H\}$.
\item[(b)] $\pi: H\to G$ is a cellular cover.
\end{enumerate}
\end{theorem}
\begin{proof}
Let $K=\Ker \pi$. By Proposition \ref{reducction-injective-surjective} we can assume
that we have a short exact sequence
$$1\to K\stackrel{i}{\longrightarrow} H\stackrel{\pi}{\longrightarrow} G\to 1.$$
First remark that if $\coGal(\pi)=1$ then $i(K)$ is central in $H$.
Indeed, let $x\in K$ and define $\psi: H\to H$ by $\psi(y)= i(x) y i(x)^{-1}$.
Then, $\pi \psi = \pi$ and $\coGal(\pi)=1$ implies $\psi=\Id_H$, i.e. $i(K)$
central in $H$.

Let $\psi: H\to K$ be any homomorphism. Then define $\psi': H\to H$ by
$\psi'(x)= x i(\psi(x))$, which is a well defined homomorphism because $i(K)$
is central. Composing by $\pi$ we get $\pi \psi' = \pi$, hence again
by hypothesis $\psi'=\Id_H$ which says precisely that $\psi=1$.

The other implication is obvious.
\end{proof}

As we will prove in Theorem~\ref{trivial-galois}, the converse of Proposition \ref{localizations-are-envelopes} for envelopes holds if we assume that
$H$ is abelian or $G$ is nilpotent.
We do not know whether or not it remains true if $H$ is nilpotent (even for nilpotency class 2).
Let $ZG$ be the center of $G$ and  $\Gamma^r G$ be its lower central series, defined inductively by $\Gamma^2G=[G,G]$, $\Gamma^{r+1}G=[\Gamma^r G,G]=[G,\Gamma^rG]$.

We use the convention for commutators $[x,y]:=xyx^{-1}y^{-1}$, and for the iterated left-normed commutator $[y_1,y_2,...,y_r]:=[...[[y_1,y_2],...,y_{r-1}],y_r] \in \Gamma^{r}G$. Recall the following identities $[a,xy]=[a,x][x,[a,y]][a,y]$ and $[xy,z]=[y,z][[y,z],x][x,z].$ Combininig these we have:
$$[abc, z_1,\cdots, z_j] =[[(ab)c,z_1],z_2,\cdots,z_j]= [[c,z_1]\;\; [[c,z_1],  ab] \;\; [ab,z_1], z_2,\cdots, z_j]=
$$
\begin{equation}
\label{abc}
[[c,z_1]    \;\;      [[c,z_1],a]  \;\;   [a, [[c,z_1], b]]  \;\;   [[c,z_1],b]  \;\;   [b,z_1]  \;\;    [[b,z_1],a]  \;\;     [a,z_1], z_2, \cdots, z_j]
\end{equation}
Let $Z_j =Z_jG$ be the $j$-term of the upper lower central series of $G$, where $Z_1=ZG$, and $Z_{j+1}=Z(G/Z_j)$.
Note that $x\in Z_j$ if and only if $[x, G,\stackrel{(j)}{\cdots},G]=1$. We call these elements {\it $j$-central}.
For later use, we recall the following known properties:
\begin{lemma} 
\label{centrals} 
Let $j\geq 1$ and $a, b, c$ elements of a group $G$.
\begin{enumerate}
\item If $b\in Z_j G$, then $[abc,z_1,\cdots, z_j]= [ac,z_1,\cdots, z_j]$
\item  If $b\in Z_{j+1}G$, then $[abc,z_1,\cdots, z_j]= [acb,z_1,\cdots, z_j]=[[ac,z_1][b,z_1],z_2, \cdots, z_j]=\cdots=[ac,z_1,\cdots,z_j][b,z_1,\cdots, z_j]$
\end{enumerate}
\end{lemma}
\begin{proof}
We proceed by induction on $j$. For $j=1$ both (1) and (2) are obvious.
Suppose by induction that (1) holds for all $j$-central elements, and suppose that $b$ is $(j+1)$-central. 
Then $ [[c,z_1], b]$, and $[b,z_1]$ $\in Z_j$, and by induction they can be eliminated in (\ref{abc}):
$$[abc, z_1,\cdots, z_{j+1}]= 
[[c,z_1]  \;\;    [[c,z_1],a]   \;\;     [a,z_1], z_2, \cdots, z_{j+1}]= [[ac, z_1],z_2, \cdots, z_{j+1}]=$$ $$=[ac, z_1,z_2, \cdots, z_{j+1}].
$$
To see (2), let us assume that it holds for all $j$-central elements, and suppose that $b$ is $(j+1)$-central.
Now in (\ref{abc}) we have that $[[c,z_1],b]$ and $[b,z_1]$ are $j$-central. By induction,
$$
[abc,z_1,\cdots, z_j]=[[c,z_1]  \;\;    [[c,z_1],a]  \;\;  [a,z_1]  \;\;  [[c,z_1],b] \;\; [b,z_1], z_2, \cdots, z_j]=
$$
\begin{equation}
\label{commutator-abc}
=[[c,z_1]  \;\;    [[c,z_1],a]  \;\;  [a,z_1]  \;\;  [b,z_1], z_2, \cdots, z_j]
\end{equation}
In the last equality, we have used that $[[c,z_1],b] \in Z_{j-1}$, and hence it can be eliminated.
Indeed, to see that, we apply induction of (2) with the elements $[b^{-1}, z_1^{-1}]$ and $[c^{-1},b] \in Z_jG$:
$$
[ [[c,z_1],b], z_2,\cdots, z_j] =[[c,z_1]b[c,z_1]^{-1}b^{-1}, z_2,\cdots, z_j] 
=[ cz_1c^{-1}z_1^{-1}bz_1cz_1^{-1}c^{-1}b^{-1}, z_2,\cdots, z_j]
$$
$$
=[cz_1c^{-1}b[b^{-1},z_1^{-1}]cz_1^{-1}c^{-1}b^{-1}, z_2,\cdots, z_j]
= [cz_1[c^{-1},b] z_1^{-1}bz_1z_1^{-1}c^{-1}b^{-1}, z_2,\cdots, z_j]
$$
$$
= [cz_1z_1^{-1}c^{-1}bcb^{-1}bc^{-1}b^{-1}, z_2,\cdots, z_j]
= \cdots=1
$$

On the other hand,
$$ [(ac)b, z_1, \cdots, z_j]=[ [b,z_1]\;\; [[b,z_1],ac] \;\; [ac,z_1], z_2, \cdots, z_j]
$$
$$
[ [b,z_1]\;\; [[b,z_1],a] \;\;   [  a, [[b,z_1],c] ] \;\;  [[b,z_1],c] \;\; [c,z_1] \;\; [[c,z_1],a] \;\;  [a,z_1], z_2, \cdots, z_j].
$$
\begin{equation}
\label{commutator-acb}
[ [b,z_1] \;\;  [c,z_1] \;\; [[c,z_1],a] \;\;  [a,z_1], z_2, \cdots, z_j]
\end{equation}
Theferore,  (\ref{commutator-abc}) and (\ref{commutator-acb}) coincide, and we have $$[abc,z_1,\cdots, z_j]=[acb,z_1,\cdots,z_j] = [[ac,z_1][b,z_1],z_2,\cdots,z_j].
$$
Applying that $[b,z_1, \cdots, z_i]$ is $(j+1-i)$-central, we obtain
$$
[abc,z_1,\cdots, z_j]= [ac,z_1,\cdots,z_j][b,z_1,\cdots,z_j].
$$
 \end{proof}

We adapt this lemma to our purposes.
\begin{lemma}\label{lema-homo} Let $X, X', Y$, $a$, $z$ be elements of a group $G$.
Given $j\geq 1$, if $[ Y,[a,X']]$ is $j$-central and $[a,Y]$ is $(j+1)$-central, then, for all $z_1,...,z_j\in G$,
$$[a,XYX',z_1,...,z_j] = [[a,XX'][a,Y],z_1,...,z_j]= [a,XX',z_1,...,z_j] [a,Y,z_1,...,z_j].$$
\end{lemma}
\begin{proof}
We have the following chain of equalities
$$[a,XYX',z_1,...,z_j]= [[a,XYX'],z_1,...,z_j]=$$
$$=[[a, X] \, [X, [a,Y]] \,  [a, Y] \,   [ Y,[a,X']] \,   [ [Y , [a,X']],X] \,   [ X ,[a,X']] \, [a,X'],z_1,...,z_j]=(*), $$
now we apply Lemma~\ref{centrals}  to get
$$
(*)=[[a, X] \,    [ X ,[a,X']] \, [a,X'] [a,Y],z_1,...,z_j]$$
$$=[[a,XX'][a,Y],z_1,...,z_j]= [a,XX',z_1,...,z_j] [a,Y,z_1,...,z_j]
$$
\end{proof}

\begin{lemma}\cite[Proposition 3.7]{Lib00}
\label{Libman-Lemma} Suppose that $G$ is nilpotent of nilpotency class  at most~$r$, i.e. $\Gamma^{r+1}G=1$. Then  for any choice of elements $a_1, \cdots, a_{r-1}\in G$ and any choice of $i$ $(1\leq i\leq r-1)$, the assignment  $x\mapsto [a_1,\cdots, a_i, x,a_{i+1}, \cdots, a_{r-1}]$ defines a homomorphism $G\to Z(G)$.\qed
\end{lemma}
This result was used in \cite{Lib00} to prove that if $\eta: H \to G$ is a localization and $G$ is nilpotent, then the nilpotency classes of $G$ and $\eta(H)$ coincide. The proof of our Theorem \ref{trivial-galois} needs the following:

\begin{lemma}
\label{repetition}
Let $\eta: H\to G$ be an envelope with $\Gal(\eta)=\{\Id_G\}$.
Let $\psi_1, \psi_2:H\to G$ be two homomorphisms such that
$\psi_1\eta=\psi_2\eta$. Suppose that, for some $i,j \geq 0$ and all $x\in G$,
\begin{equation}\label{imas1j} [\Gamma^{i+1}G, \psi_1(x)\psi_2(x)^{-1}, G, \stackrel{(j}{\cdots}, G=1\end{equation} and \begin{equation}\label{ijmas1}
[\Gamma^{i}G, \psi_1(x)\psi_2(x)^{-1}, G, \stackrel{(j+1}{\cdots}, G]=1.\end{equation}.
Then, $[\Gamma^{i}G, \psi_1(x)\psi_2(x)^{-1}, G, \stackrel{(j}{\cdots}, G]=1$, for all $x\in G$.
\end{lemma}
\begin{proof} Fix $a\in \Gamma^iG$ and $z_1,\cdots, z_j \in G$.
Let $\xi:G\to G$ be the map given by
$$\xi(x):=x[a, \psi_1(x)\psi_2(x)^{-1},  z_1, \cdots, z_j].$$
By assumption (\ref{ijmas1}), $[a, \psi_1(x)\psi_2(x)^{-1}, z_1, \cdots, z_j]$ is central.
We next show that $\xi$ is a homomorphism. Given $x$ and $y$ of $G$, we have
$$
\xi(xy)=xy\underbrace{[a, \psi_1(x)\psi_1(y)\psi_2(y)^{-1}\psi_2(x)^{-1}, z_1, \cdots, z_j]}.
$$
We apply Lemma~\ref{lema-homo} to $X= \psi_1(x)$, $Y=\psi_1(y)\psi_2(y)^{-1}$ and $X'=\psi_2(x)^{-1}$. The underbraced part is
$$
[a, XYX',  z_1, \cdots, z_j] = [[a, XYX'], z_1, \cdots, z_j]=(*)
$$
The element 
$[Y,[a,X']]$ is $j$-central because of (\ref{imas1j}), 
and $[a, Y]$ is $(j+1)$-central, by (\ref{ijmas1}). Therefore we can apply Lemma~\ref{lema-homo} to get:
$$
(*)=[[a, XX'] [a, Y],  z_1, \cdots, z_j] = [a, XX',  z_1, \cdots, z_j] [a, Y, z_1, \cdots, z_j].
$$
This proves that $\xi$ is a homomorphism, which satisfies $\xi\eta=\eta$ because $\psi_1\eta=\psi_2\eta$, hence
 $\xi\in \Gal(\eta)=\{\Id_G\}$. This concludes
$[\Gamma^{i}G, \psi_1(x)\psi_2(x)^{-1}, G, \stackrel{(j}{\cdots}, G]=1$, for all $x\in G$.
\end{proof}
\begin{theorem}
\label{trivial-galois}
Let $\eta: H\to G$ be an envelope. Suppose that $H$ is abelian or $G$ is nilpotent. Then the following are equivalent:
\begin{enumerate}
\item[(a)] $\Gal(\eta)=\{{\rm Id}_G\}$.

\item[(b)] $\eta: H\to G$ is a localization.

\end{enumerate}
\end{theorem}
\begin{proof} Proposition \ref{localizations-are-envelopes} gives that (b) implies (a).
To prove that (a) implies (b), let $\psi_1$ and $\psi_2$ be endomorphisms of $G$ such that $\psi_1 \eta = \psi_2\eta$. We want to show that $\psi_1=\psi_2$. This equality easily follows if the element $\psi_1(x)\psi_2(x)^{-1}$ is central. Indeed, in that case the map $\xi: x\mapsto x\psi_1(x)\psi_2(x)^{-1}$  defines an element of $\Gal(\eta)$. Hence $\xi={\rm Id}_G$ by (a), which implies $\psi_1=\psi_2$ as desired.

Case $H$ abelian: For each $h\in H$, we have $c_{\eta (h)} \eta=\eta$ which forces $c_{\eta(h)}={\rm Id}_G$ by (a). Therefore, for each $g\in G$, $c_g{\eta(h)} = g \eta(h) g^{-1}= \eta(g)$,  that is,  $c_g\eta = \eta$. Again by (a), it follows $c_g={\rm Id}_G$ for all $g\in G$, that is,  $G$ is abelian.

Case $G$ nilpotent of class at most $r\geq 2$ (that is $[G, \stackrel{(r+1}{\cdots}, G]=1$): In particular, by Lemma~\ref{repetition} we get 
$$[ G, \stackrel{(i}{\cdots}, G, \psi_1(x)\psi_2(x)^{-1}, G, \stackrel{(j}{\cdots}, G]=1,$$
whenever $i+j\geq r-1$.
Iterating this argument we get that $[\psi_1(x)\psi_2(x)^{-1},G]=1$. Thus $\psi_1(x)\psi_2(x)^{-1}$ is central.
\end{proof}







\section{Some envelopes and open questions}
\label{examples}
We next give some examples of envelopes, results and open questions.
We deal with envelopes according to Proposition \ref{charac-env}, having in mind related examples and problems of localizations of groups. Although envelopes do not share as many good properties as localizations, we will still have interesting questions to treat.
Recall for example that if $H$ is abelian and  $H \hookrightarrow G$ is a localization, then $G$ is abelian
  (see e.g. \cite{Lib00} or \cite[Theorem 2.2]{Cas00}). This fails for arbitrary envelopes.
\begin{example}
\label{dihedral-group}
{\rm
The inclusion $i: C_p\hookrightarrow D_{2p}$
of the cyclic group of order a prime $p>2$ into the
dihedral group of order $2p$ is an envelope,
with $\Gal(i)=C_p$.
}
\end{example}

Group localizations of $\Z$
(and more generally, of any commutative ring with unit) have
attracted much attention in the literature.
They were called {\it $E$-rings} by Schultz (see \cite{Sch}).
These are commutative rings $A$ with identity 1, for which the evaluation at~1,
is an isomorphism $\End(A)\cong A$ of rings.
The following is an envelope of $\Z$, similar to \ref{dihedral-group}.

\begin{example}
{\rm
Consider $G =\langle  x, y : y^2=1, x^y=x^{-1} \rangle$ the infinite dihedral group. Then $\langle x \rangle \hookrightarrow G$ is an envelope, with $\Z$ as its Galois group.
}
\end{example}
One could try to construct either countable or arbitrarily large envelopes of $\Z$
with prescribed Galois group $C_p$ (or another fixed group), either using Corner's techniques or
some adjusted Shelah's black box.

\begin{question}
{\rm Are there arbitrarily large envelopes of $\Z$ with any prescribed Galois group?}
\end{question}
Of course, variations of this problem can be treated in parallel to constructions of absolute $E$-rings \cite{GHS11}, or the most recent $\aleph_k$-free $E$-rings.

Given a localization $H \to G$ of a nilpotent group $H$,
is $G$ nilpotent? This is still an open and very difficult problem in localization theory.
We recall that this is only known to be true for nilpotent groups $H$ of nilpotency class 2 (Libman \cite[Theorem 3.3]{Lib00}, Casacuberta \cite[Theorem 2.3]{Cas00}) and also for finite $p$-groups of nilpotency class 3 by Aschbacher \cite{Aschbacher}.
This problem is even open for finite $p$-groups $H$, where $p$ is any prime.
For envelopes we think that the following question would be interesting
 and perhaps easy to solve.

\begin{question}
{\rm
When is an inclusion $H \hookrightarrow G$ of finite $p$-groups an envelope? Characterize them, give examples.
}
\end{question}


We know from \cite{RSV05} that if $H\to G$ is a localization
between two finite groups, and $H$ is perfect, then $G$ is perfect too.
This is no longer true for envelopes.

\begin{example}
{\rm
The inclusion  $A_n \hookrightarrow S_n$ of the alternating group into the symmetric group
is an envelope for $n\geq 5$.
}
\end{example}
In \cite{RST02} it is given a criterium that permits to know when an inclusion
of finite simple groups is a localization. For envelopes we have
the following characterization.
\begin{theorem}
\label{characterization-envelopes-finite-simple-groups}
Let $\varphi:H\hookrightarrow G$ be an inclusion of non-abelian finite simple groups.
Then $\varphi$ is an envelope if and only if:
\begin{enumerate}
\item Every automorphism $\alpha: H \ra H$ extends to a (non-necessarily unique)
automorphism
$i(\alpha): G \ra G$.
\item Any subgroup of $G$ which is isomorphic to $H$ is conjugate to
$H$ in $\Aut(G)$.
\end{enumerate}
Furthermore, $Gal(\varphi)$ is isomorphic to the centralizer subgroup of $H$ in $G$.
\end{theorem}
\begin{proof}
The proof is basically the same as in \cite[Theorem 1.4]{RST02}.
Note that for localizations the assignation $\beta \mapsto i(\beta)$
is unique and respects compositions giving rise to a well defined
homomorphism $\Aut(H) \to \Aut(G)$.
\end{proof}

This characterization restricted to complete groups has a simpler formulation (cf. \cite[Corollary 1.6]{RST02}).  Recall that a group is {\it complete} if it has trivial center and every automorphism is a conjugation.
\begin{corollary}
Let $H$ be a non-abelian simple subgroup of a finite simple group $G$ and let $i:H\hookrightarrow G$ be the inclusion. Assume that $H$ and $G$ are complete groups. Then $i$ is an envelope if and only if any subgroup of $G$ which is isomorphic to $H$ is conjugate to $H$. \qed
\end{corollary}

Notice that the equivalences of Theorem \ref{trivial-galois} also hold for finite simple groups.
\begin{corollary}
An envelope $i: H\hookrightarrow G$
between two non-abelian finite simple groups is a localization if and only if
it has trivial Galois group. \qed
\end{corollary}
In \cite{RST02} one can find many examples of envelopes,
that are not localizations, since the centralizer fails to be trivial.
The authors also considered the concept of  {\it rigid components} of non-abelian finite simple groups  defined by the following equivalence relation: two such groups are related if there is a zig-zag of  monomorphisms which are localizations connecting them. Parker--Saxl \cite{PS06}, completing results of \cite{RST02}, showed that all (non-abelian) finite simple groups lie in the same rigid component, except the ones of the form $PSp_{4}(p^{2^c})$, with $p$ an odd prime and $c>0$,
which are isolated. This fact motivates the following

\begin{question}
{\rm
Are all non-abelian finite simple groups lying in the same {\it weak rigid component}? Two groups are in the same weak rigid component if there is a zig-zag of monomorphisms connecting them which are envelopes.
}
\end{question}

G\"obel--Rodr\'{\i}guez--Shelah \cite{GRS02} and G\"obel--Shelah \cite{GS02} found out that every finite simple group admits arbitrarily large localizations.

\begin{question}
{\rm
Which finite simple groups admit arbitrarily large envelopes
with a prescribed Galois group?
}
\end{question}





\setlength{\baselineskip}{0.5cm}


\vskip 0.5 cm

\setlength{\baselineskip}{0.6cm}

Sergio Estrada

Departamento de Matem\'aticas

Universidad de Murcia

Campus de Espinardo, Espinardo (Murcia) E-30100, Murcia, Spain, e-mail: {\tt sestrada@um.es}

\bigskip\noindent

Jos\'e L. Rodr\'{\i}guez

Departamento de Matem\'aticas

Universidad de Almer\'{\i}a

La Ca\~nada de San Urbano, s/n, E-04120 Almer\'{\i}a,
 Spain, e-mail: {\tt
jlrodri@ual.es}


\begin{thebibliography}{99}
\vspace*{0.5cm}





\bibitem{Aschbacher}
M. Aschbacher, {\it On a question of Farjoun}, Finite groups 2003, 1--27, Walter de Gruyter GmbH \& Co. KG, Berlin, 2004.

\bibitem{AS80}
M. Auslander, and S. O. Smal{\o} ,
{\it Preprojective modules over Artin algebras},
J. Algebra {\bf 66} (1980),  61--122.




\bibitem{BCFS13} M. Blomgren, W. Chach\'olski, E. Dror Farjoun, and Y. Segev, {\it Idempotent transformations of finite groups}, Adv. Math. {\bf 233} (2013), 56--86.

\bibitem{Bou77} A. P. Bousfield,  {\em Homotopy facatorization systems},
 J. Pure Appl. Algebra {\bf 9} (1977), 207--220.




\bibitem{BD06} J. Buckner and M. Dugas, {\it Co-local subgroups of abelian
groups}, pp. 29--37 in
Abelian groups, rings, modules and homological algebra; Proceedings in honor of
Enochs, Lect. Notes
Pure Appl. Math. {\bf 249}, Chapman \& Hill, Boca Raton, FL 2006.



\bibitem{Cas00} C. Casacuberta, {\it On structures preserved by idempotent
transformations of groups and homotopy types\/},
Igodt, Paul (ed.) et al., Crystallographic groups and their generalizations. II. Proceedings of the workshop, Katholieke Universiteit Leuven, Campus Kortrijk, Belgium, May 26-28, 1999. Providence, RI: American Mathematical Society (AMS). Contemp. Math. 262, 39-68 (2000).




\bibitem{CRT} C. Casacuberta, J. L. Rodr\'{\i}guez, J.-Y. Tai
{\it Localization of abelian Eilenberg--Mac Lane spaces of finite type},  Algebr. Geom. Topol., to appear.






\bibitem{CFGS07} W. Chach\'olski, E. Dror Farjoun, R. G\"obel, and Y. Segev,
{\it Cellular covers of divisible abelian groups\/},
Alpine perspectives on algebraic topology, 77--97,
Contemp. Math., 504, Amer. Math. Soc., Providence, RI, 2009.








\bibitem{E81} E. Enochs,  {\em Injective and flat covers,
envelopes and resolvents}, Israel J. Math. {\bf 39} (1981), 189--209.


\bibitem{E05}
E. Enochs and J. Rada, {\it Abelian groups which have trivial absolute coGalois
group},
Czech. Math. J. {\bf 55} (130) (2005), 433--437.















\bibitem{F97} E. Dror Farjoun, \emph{Cellular spaces, null spaces and
homotopy localization,} Lecture Notes in Math. {\bf 1622}
Springer-Verlag, Berlin--Heidelberg--New York 1996.


\bibitem{FGSS07}E. Dror Farjoun, R. G\"obel, Y. Segev, and S. Shelah,
{\it On kernels of cellular covers}, Groups Geom. Dyn. {\bf 1}
(2007), no. 4, 409--419.






\bibitem{GHS11} R. G\"obel, D. Herden, and S. Shelah, {\it Absolute E-rings}, Adv. Math. {\bf 226} (2011), 235--253.

\bibitem{GRS02} R. G\"obel, J. L. Rodr\'{\i}guez, and S. Shelah, {\em Large
localizations of finite simple groups}, J. Reine angew. Math. {\bf 550} (2002),
1--24.


\bibitem{GRS12} R. G\"obel, J. L. Rodr\'{\i}guez, and L. Str\"ungmann,
\emph{Cellular covers of cotorsion-free modules},
Fund. Math. {\bf 217} (2012), 211--231.

\bibitem{GS02}
R. G\"obel and S. Shelah,
\emph{Constructing simple groups for localizations},
Comm. Algebra
{\bf 30} (2002), 809--837.


\bibitem{GT12} R. G\"obel, J.Trlifaj, Approximations and Endomorphism Algebras of Modoules, 2nd revised and extended edition, de Gruyter Expositions in Mathematics 41, Vol.1 - Approximations, Vol.2 - Predictions, xxviii + 972 pages, W. de Gruyter, Berlin - Boston 2012.





\bibitem{He11} M. H\'ebert, {\it Weak reflections and weak factorization systems},
Appl. Categorical Structures. {\bf 9} (2011), 9--38.


\bibitem{Hill} P. Hill, {\it Abelian group pairs having a trivial coGalois
group}, Czechoslovak Math. J. {\bf 58},  (2008), 1069--1081.



\bibitem{Hi03} P. S. Hirschhorn, Model Categories and Their Localizations, Math.
Surveys Monographs, vol. 99, Amer. Math. Soc. Providence, 2003.


\bibitem{Lib00} A. Libman, {\em Cardinality and nilpotency of
localizations of groups and $G$-modules}, Israel J. Math. {\bf 117} (2000), 221--237.

\bibitem{PS06} C. Parker and J. Saxl, {\it Localization and finite simple
groups}, Israel J. Math. {\bf 153} (2006),
285--305.

\bibitem{Porst81} H. E. Porst, {\it Characterization of injective envelopes},
Cahiers de topologie et geometrie diff. cat. {\bf 22} (1981), 399--406.






\bibitem{RS00}
J. L. Rodr\'{\i}guez and D. Scevenels,
{\it Universal epimorphic equivalences for groups}, J. Pure Appl.
Algebra, {\bf 148} (2000), 309--316.





\bibitem{RST02} J. L. Rodr\'{\i}guez, J. Scherer, and J. Th\'evenaz,
{\em Finite simple groups and localization},
in Israel J. Math. {\bf 131} (2002), 185--202.




\bibitem{RSV05} J. L. Rodr\'{\i}guez, J. Scherer, and A. Viruel, \emph{ Preservation of
perfectness and acyclicity. Berrick and Casacuberta's universal acyclic space
localized at a set of primes}, Forum Math. {\bf 17} (2005), 67--75.





\bibitem{Sch}
P. Schultz,
The endomorphism ring of the additive group of a ring,
{\em J. Austral. Math. Soc. Ser. A}
{\bf 15} (1973), 60--69.




\bibitem{Xu96}
J. Xu, {\it Flat Covers of Modules}, Lecture Notes in Mathematics, Vol. 1634,
Springer-
Verlag, New York/Berlin, 1996.

\end{thebibliography}
\end{document}